\documentclass[12pt]{article}

      \usepackage{latexsym}
         \usepackage[reqno, namelimits, sumlimits]{amsmath}
         \usepackage{amssymb, amsfonts}
         \usepackage{amsthm}
           
 \newtheorem{theorem}{Theorem}[section]
 
 \newtheorem{definition}[theorem]{Definition}
 
 \newtheorem{lemma}[theorem]{Lemma}
 
 \newtheorem{remark}[theorem]{Remark}
 
 \newtheorem{cor}[theorem]{Corollary}
 \newtheorem{pro}[theorem]{Proposition}

\author{ Gregory Seregin
}

\title{A Slightly Supercritical Condition of Regularity of Axisymmetric Solutions to the Navier-Stokes Equations}

\author{G.~Seregin\footnote{University of Oxford, Mathematical Institute, OxPDE, Oxford, UK and St Petersburg Department of Steklov Mathematical Institute, RAS, Russia, email address: \texttt{seregin@maths.ox.ac.uk}}
}

\begin{document}

\maketitle


\centerline{Dedicated to Yoshihiro Shibata
}

\begin{abstract} In the note, a new 
regularity condition for axisymmetric solutions to the non-stationary 3D Navier-Stokes equations is proven. It is slightly supercritical.
\end{abstract}


{\bf Data availability statement}
Data sharing not applicable to this article as no datasets were generated or analysed during the current study.

{\bf Acknowledgement} 
 The work is supported by the grant RFBR 20-01-00397.

\pagebreak

\setcounter{equation}{0}
\section{Introduction }

In this note, we continue to analyse   
potential singularities of axisymmetric solutions to the non-stationary Navier-Stokes equations. In the previous paper \cite{Seregin2020}, it has been shown that  an axially symmetric solution is smooth provided a certain scale-invariant energy quantity of the velocity field is bounded. 
By definition, a potential singularity with bounded scale-invariant energy quantities is called the Type I blowup. It is important to notice that the above  result does not follow from the so-called $\varepsilon$-regularity theory developed in \cite{CKN}, \cite{Lin}, and \cite{LS1999}, where regularity is coming out due to  smallness of those scale-invariant energy quantities.

We consider the 3D Navier-Stokes system
\begin{equation}\label{NSE}
\partial_tv+v\cdot\nabla v-\Delta v=-\nabla q,\qquad {\rm div}\,v=0	
\end{equation}
in the parabolic cylinder $Q=\mathcal C\times]-1,0[$, where $\mathcal C=\{x=(x_1,x_2,x_3):\,x_1^2+x^2_2<1,\,-1<x_3<1\}$.
A solution $v$ and $q$ is supposed to be  a suitable weak one, which means the following:
\begin{definition}\label{sws}
Let $\omega\subset \mathbb R^3$ and $T_2>T_1$. The pair $w$ and $r$ is a suitable weak solution to the Navier-Stokes system in $Q_*=\omega\times ]T_1,T_2[$ if:

1. $w\in L_{2,\infty}(Q_*)$, $\nabla w\in L_2(Q_*)$, $r\in L_\frac 32(Q_*)$;

2. $w$ and $r$ satisfy the Navier-Stokes equations in $Q_*$ in the sense of distributions;

3. for a.a. $t\in [T_1,T_2]$, the local energy inequality
$$\int\limits_\omega\varphi(x,t)|w(x,t)|^2dx+2
\int\limits_{T_1}^t\int\limits_\omega\varphi|\nabla w|^2dxdt'\leq\int\limits_{T_1}^t\int\limits_\omega[|w|^2(\partial_t\varphi+\Delta \varphi)+$$
$$+w\cdot\nabla\varphi(|w|^2+2r)]dxdt'$$	holds for all non-negative $\varphi\in C^1_0(\omega\times ]T_1,T_2+(T_2-T_1)/2[).$
\end{definition}

In our standing assumption, it is supposed that a suitable weak solution $v$ and $q$ to the Navier-Stokes equations in $Q=\mathcal C\times ]-1,0[$ is axially symmetric with respect to the axis $x_3$. The latter means the following: if we introduce the corresponding cylindrical coordinates $(\varrho,\varphi,x_3)$ and use the corresponding representation $v=v_\varrho e_\varrho+v_\varphi e_\varphi+v_3e_3$, then $v_{\varrho,\varphi}=v_{\varphi,\varphi}=v_{3,\varphi}=q_{,\varphi}=0$.

There are many papers on regularity of axially symmetric solutions. We cannot pretend to cite all good works in this direction. For example, 
let us mention papers: \cite{L1968}, \cite{UY1968},
\cite{LMNP1999}, \cite{NP2001}, \cite{Po01}, \cite{CL2002}, \cite{SZ2007}, \cite{ChenStrainYauTsai2009},
\cite{SS2009},  \cite{LeiZhang2011}, \cite{Pan16},\cite{LeiZhang2017} \cite{ChenFangZhang2017},  \cite{SerZhou2018}, and \cite{ZhangZhang2014}.

Actually, our note is inspired by the paper \cite{Pan16}, where the regularity of solutions has been proved under a slightly supercritical assumption. We would like to consider a different supercritical assumption, to give a different proof and to get a better result.

To state our supercritical assumption, additional notation is needed. Given $x=(x_1,x_2,x_3)\in \mathbb R^3$, denote $x'=(x_1,x_2,0)$. Next, different types of cylinders will be denoted as $\mathcal C(r)=\{x:\,|x'|<r, |x_3|<r\}$, $\mathcal C(x_0,r)=\mathcal C(r)+x_0$, $Q^{\lambda,\mu}(r)=\mathcal C(\lambda r)\times]-\mu R^2,0[$, $Q^{1,1}(r)= Q(r)$, $Q^{\lambda,\mu}(z_0,r)=\mathcal C(x_0,\lambda r)\times]t_0-\mu R^2,t_0[$.  And, finally, we let
$$
f(R):=\frac{1}{\sqrt R} \Big(\int\limits^0_{-R^2}\Big(\int\limits_{\mathcal C(R)}	|v|^3dx\Big)^\frac 43dt\Big)^\frac 34
$$
and 
$$
	M(R):=\frac{1}{\sqrt R}\Big(\int\limits_{Q(R)}|v|^\frac {10}3dz\Big)^\frac 3{10}$$
for any $0<R\leq  1$ and  assume that:
\begin{equation}
	\label{sca1}
	f(R)+M(R)\leq g(R):=c_*\ln^\alpha\ln^\frac 12(1/R)
\end{equation}
for all $0<R\leq 2/3$, where $c_*$ and $\alpha$ are positive constants and $\alpha$ obeys the condition:
\begin{equation}
	\label{sca2} 0<\alpha\leq \frac 1{224}.
\end{equation}

Without loss of generality, one may assume that $g(R)\geq 1$ for $0<R\leq \frac 23$. To ensure the above condition, it is enough to increase the constant $c_*$ if necessary.

Our aim could be the following completely local statement. 
 \begin{theorem}
\label{mainresult}
Assume that a pair $v$ and $q$	is axially symmetric suitable weak solution to the Navier-Stokes equations in $Q$ 
and conditions (\ref{sca1}) and (\ref{sca2})
hold. Then
the origin $z=0$ is a regular point of $v$. 
\end{theorem} 

However, in this paper, we shall prove a weaker result leaving Theorem \ref{mainresult} as a plausible conjecture. We shall return to a proof of Theorem \ref{mainresult} elsewhere. In the present paper, the following fact is going to be justified.
\begin{theorem} \label{mainresultglobal}
Let $v$ be an  axially symmetric solution to the Cauchy problem for the Xavier-Stokes equations (\ref{NSE}) in $\mathbb R^3\times ]0,T[$ with initial divergence free field $v_0$ from the Sobolev space $H^2 =W^2_2(\mathbb R^3)$ such that 
$$\sup\limits_{0<t<T-\delta}\|\nabla v(\cdot,t)\|_{L_2(\mathbb R^3)}\leq C(\delta)<\infty$$
for all $0<\delta <T$. 
Assume further that 
\begin{equation}
	\label{iniswirl} \Sigma_0=
	\sup\limits_{x\in \mathbb R^3}|v_{02}(x)x_1-v_{01}(x)x_2|<\infty
\end{equation}
and
\begin{equation}\label{globalsca1}
\sup\limits_{0<R\leq 2/3}\sup\limits_{-\infty<h<\infty} f(R;(0,h,T))+M(R;(0,h,T))\leq g(R)	
\end{equation}
with some positive constants $c_*$ and $\alpha $, satisfying (\ref{sca2}), where 
$$f(R;z_0):=\frac{1}{\sqrt R} \Big(\int\limits^{t_0}_{t_0-R^2}\Big(\int\limits_{\mathcal C(x_0,R)}	|v|^3dx\Big)^\frac 43dt\Big)^\frac 34
$$
and 
$$
	M(R;z_0):=\frac{1}{\sqrt R}\Big(\int\limits_{Q(z_0,R)}|v|^\frac {10}3dz\Big)^\frac 3{10}.$$
	
	Then $v$ is a strong solution to the above Cauchy problem in $\mathbb R^3\times ]0,T[$, i.e.,
	$$\sup\limits_{0<t<T}\|\nabla v(\cdot,t)\|_{L_2(\mathbb R^3)}<\infty.$$
\end{theorem}


Our proof is based on the analysis of the following scalar equation
\begin{equation}
	\label{eqsigma}
	\partial_t\sigma+\Big(v+2\frac{x'}{|x'|^2}\Big)\cdot\nabla\sigma-\Delta\sigma=0
\end{equation}
in $Q\setminus(\{x'=0\}\times]-1,0[)$, where
 $\sigma :=\varrho v_\varphi=v_2x_1-v_1x_2$.
 
 Let us list some differentiability properties of $\sigma$. Some of them follows from partial regularity theory developed by Caffarelli-Kohn-Nirenberg. 
 
 Indeed,
 since $v$ and $q$ are an axially symmetric suitable weak solution, there exists a closed set $S^\sigma$ in $Q$, whose 1D-parabolic measure in $\mathbb R^3\times \mathbb R$ is equal to zero and $x'=0$ for any $z=(x,t)\in S^\sigma$, such that any spatial derivative of $v$ (and thus of $\sigma$) is H\"older continuous in 
 $Q\setminus S^\sigma$. 
 
 Next, we observe that
$$|\partial_t\sigma(z)-\Delta \sigma(z)|\leq (\sup\limits_{z=(x,t)\in P(\delta,R;R)\times ]-R^2,0[}|v(z)|+2/\delta)|\nabla \sigma(z)|
$$ for any $0<\delta<R<1$, where $P(a,b;h)=\{x: \,a<|x'|<b,\,|x_3|<h\}$.
Since $v$ is axially symmetric, the first factor on the right hand side is finite. This fact, by iteration, yields
$$\sigma\in W^{2,1}_p(P(\delta,R;R)\times ]-R^2,0[)$$
for any $0<\delta <R<1$ and for any finite $p\geq 2$. 

It follows from the above partial regularity theory that, for any $-1<t<0$,
\begin{equation}
	\label{limit}
	\sigma(x',x_3,t)\to 0 \quad\mbox{as}\quad |x'|\to 0
\end{equation}
for all $x_3\in]-1,1[\setminus S^\sigma_t$.

In the same way, as it has been done in \cite{SS2009} and \cite{Seregin2020},
one can show that
$\sigma\in L_\infty(Q(R))$ for any $0<R<1$.

The main part of the proof of Theorem \ref{mainresultglobal} is  the following  fact. 

\begin{pro}
\label{modcon}
 Let $\sigma=\varrho v_\varphi$, then
 \begin{equation}
\label{osc}
{\rm osc}_{z\in Q(r)}\sigma\leq C_1(c_*)\Big(\frac r{2R}\Big)^{C_2(c_*)}{\rm osc}_{z\in Q(2R)}\sigma(z),\end{equation}
	where $C_1$ and $C_2$ are positive constants and $0<r<R\leq  R_*(c_*,\alpha)\leq 1/6$. Here, ${\rm osc}_{z\in Q(r)}\sigma(z) =M_{r}-m_{r}$ and
$$M_{r}=\sup\limits_{z\in Q(r)}\sigma(z),\qquad m_{r}=\inf\limits_{z\in Q(r)}\sigma(z).$$ 	\end{pro}

The above statement is an  improvement of the result in \cite{Pan16}, where the bound for oscillations of $\sigma$ contains a logarithmic factor only.

The proof of Proposition \ref{modcon} is based on a technique developed in    \cite{NazUr}, see also references there. We also would like to mention interesting results 
for the heat equation with a divergence free drift, see \cite{Ign2014}, \cite{IKR16}, \cite{SSSZ2012}, and \cite{AH2021}.


\setcounter{equation}{0}
\section{Auxiliary Facts}

Define the  class $\mathcal V$ of functions $\pi:Q\to \mathbb R$ possessing the  properties:

\noindent
(i) there exists a closed set $S^\pi$ in $Q$, whose 1D-parabolic measure $\mathbb R^3\times \mathbb R$ is equal to zero and $x'=0$ for any $z=(x',x_3,t)\in S^\pi$, such that any spatial derivative is H\"older continuous in $Q\setminus S^\pi$; 

\noindent
(ii) $$\pi\in W^{2,1}_2(P(\delta,R;R)\times ]-R^2,0[)\cap L_\infty(Q(R))$$
for any $0<\delta <R<1$.

We are going to use the following subclass $\mathcal V_0$ of the class $\mathcal V$, saying that $\pi\in \mathcal V_0$ if and only if $\pi\in \mathcal V$ and 
\begin{equation}\label{GammaEq2}
\partial_t\pi+\Big(u+2\frac {x'}{|x'|^2}\Big)\cdot\nabla \pi-\Delta \pi= 0\end{equation}
in $\mathcal C\setminus \{x'=0\}\times ]-1,0[$.

We shall also say  that $\pi\in \mathcal V_0$  
 has the property $(\mathcal B_R)$ in $Q(2R)$ if there exists  a number $k_R>0$ such that $\pi(0,x_3,t)\geq k_R$ for $-(2R)^2\leq t\leq 0$, $x_3\in ]-2R,2R[\setminus S^\pi_t$, where 
 $S^\pi_t=\{x\in \mathcal C: (x,t)\in S^\pi\}$.
\begin{remark} \label{propertyBR}
Let $0<r\leq R$ and $\pi\in \mathcal V_0$  
 have the property $(\mathcal B_R)$ in $Q(2R)$. 
 Then $\pi$ has the property $(\mathcal B_r)$ in $Q(2r)$	with any constant less or equal to $k_R$.
 \end{remark}
 
In what follows, we always suppose that $0<R\leq1/6$.

\begin{pro}
	\label{Mozer} Let $\pi\in\mathcal V_0$ have the property $(\mathcal B_R)$. Then, for any $0<k\leq k_R$, for any $0<\tau_1<\tau<2$, and for any $0<\gamma_1<\gamma<4$, the following inequality holds: 
	\begin{equation}\label{Mozer1}
		\sup\limits_{z\in Q^{\tau_1,\gamma_1}(R)}\sigma(z)\leq 
		c_1(\tau_1,\tau,\gamma_1,\gamma,M(2R))\Big(\frac 1{|Q^{\tau,\gamma}(R)|}\int\limits_{Q^{\tau,\gamma}(R)}\sigma^\frac {10}3(R)dz\Big)^\frac 3{10},	\end{equation}
where $\sigma =(k-\pi)_+$,
$$c_1(\tau_1,\tau,\gamma_1,\gamma,M(2R))=\frac c{(\tau-\tau_1)^\frac {16}3}\Big(1+\frac {\tau-\tau_1}{\sqrt{\gamma-\gamma_1}}+\Big(\frac 1{\gamma_1\tau_1^3}\Big)^\frac 1{10}M(2R)\Big)^3,$$
and $Q^{\tau,\gamma}(R)=\mathcal C(\tau R)\times ]-\gamma R^2,0[$.	
\end{pro}
\begin{proof} Repeating arguments in \cite{Seregin2020}, we can get the following estimate of $h=\sigma^m$:
$$\Big(\int\limits^0_{t_2}\int\limits_{\mathcal C(r_2)}|h|^\frac{10}{3}dz\Big)^\frac 3{10}\leq$$ 
\begin{equation}
	\label{startingMozer}
	\leq c\Big(\int\limits^0_{t_1}\int\limits_{\mathcal C(r_1)}|h|^\frac{5}{2}dz\Big)^\frac 2{5}
	\frac {(r_1^3|t_1|)^\frac 1{10}}{r_1-r_2}
	\Big(1+\frac {r_1-r_2}{\sqrt{t_2-t_1}}+\overline M(r_1,t_1)+\frac {r_1^\frac {13}8|t_1|^\frac 1{36}}{(r_1-r_2)^\frac 79}\Big)\end{equation}
for any $0<r_2<r_1<2R$ and $-4R^2<t_1<t_2<0$, where
$$\overline M(r_1,t_1)=\Big(\frac 1{|t_1|r_1^3}\Big)^\frac 1{10}\Big(\int\limits^0_{t_1}\int\limits_{\mathcal C(r_1)}|v|^\frac{10}{3}dz\Big)^\frac 3{10}.$$

Next, we wish to iterate  (\ref{startingMozer}). To this end,     let $m=m_i=\Big( 4/3\Big)^i$, 
$$r_1=r_i=\tau_1R+(\tau-\tau_1)R2^{-i+1}, \qquad r_2=r_{i+1}
,$$	
$$t_1=t_i=-\gamma_1R^2-(\gamma-\gamma_1)R^24^{-i+1},\qquad
t_2=t_{i+1},$$
where $i=1,2,...$. Then, we can derive  from (\ref{startingMozer}) the following inequality
\begin{equation}
\label{iter1}G_{i+1}\leq \Big(\frac {c2^{i+1}}{\tau-\tau_1}\Big)^\frac 1{m_i}\Big(1+\frac {\tau-\tau_1}{\sqrt{\gamma-\gamma_1}}+\overline M(r_i,t_i)+\frac {2^{(i+1)\frac 79}}{(\tau-\tau_1)^\frac 79}\Big)^\frac 1{m_i}G_i,
\end{equation}
where
$$G_i=\Big( \frac 1{|t_i|r_i^3}\int\limits^0_{t_i}\int\limits_{\mathcal C(r_i)}\sigma^\frac {5m_i}2dz\Big)^\frac 2{5m_i}.
$$
Noticing that 
$$\overline M(r_i,t_i)\leq c\Big(\frac 1{\gamma_1\tau_1^3}\Big)^\frac1{10}M(2R),
$$
let us make use of  (\ref{iter1}) to obtain the estimate
$$G_{i+1}\leq$$
\begin{equation}
\label{iter2}\leq \Big(\frac {c2^{i+1}}{\tau-\tau_1}\Big)^\frac 1{m_i}\Big(1+\frac {\tau-\tau_1}{\sqrt{\gamma-\gamma_1}}+\frac {2^{(i+1)\frac 79}}{(\tau-\tau_1)^\frac 79}+\Big(\frac 1{\gamma_1\tau_1^3}\Big)^\frac1{10}M(2R)\Big)^\frac 1{m_i}G_i,
\end{equation}
which, after iterations, gives the following
\begin{equation}
\label{iter3}
G_{i+1}\leq \xi_iG_1,
\end{equation}
where
$$\xi_i=\prod\limits^i_{k=1}\Big(\frac {c2^{k+1}}{\tau-\tau_1}\Big)^\frac 1{m_k}\Big(1+\frac {\tau-\tau_1}{\sqrt{\gamma-\gamma_1}}+\frac {2^{(k+1)\frac 79}}{(\tau-\tau_1)^\frac 79}+\Big(\frac 1{\gamma_1\tau_1^3}\Big)^\frac1{10}M(2R)\Big)^\frac 1{m_k}.$$
Obviously,
$$\xi_i\leq \prod\limits^i_{k=1}\Big(\frac {c2^{k+1}}{\tau-\tau_1}\Big)^\frac 1{m_k}
\Big(1+\frac {2^{(k+1)\frac 79}}{(\tau-\tau_1)^\frac 79}\Big)^\frac 1{m_k}\Big(1+\frac {\tau-\tau_1}{\sqrt{\gamma-\gamma_1}}+$$$$+\Big(\frac 1{\gamma_1\tau_1^3}\Big)^\frac1{10}M(2R)\Big)^\frac 1{m_k}.$$
Next,
$$\ln\xi_i\leq A_1+A_2+A_3,
$$
where
$$A_1=\sum^i_{k=1}\frac 1{m_k}(\ln c +(k+1)\ln 2-\ln (\tau-\tau_1))\leq\ln c-3\ln (\tau-\tau_1),$$
$$A_2=\sum^i_{k=1}\frac 1{m_k}\ln\Big(1+\frac {2^{(k+1)\frac 79}}{(\tau-\tau_1)^\frac 79}\Big)=\sum^i_{k=1}\frac 1{m_k}\ln \Big(\frac {2^{(k+1)\frac 79}}{(\tau-\tau_1)^\frac 79}\Big)+$$
$$+\frac 1{m_k}\ln\Big(1+\frac{(\tau-\tau_1)^\frac 79}{2^{(k+1)\frac 79}}\Big)\leq \ln \frac c{(\tau-\tau_1)^\frac 73}+$$
$$+(\tau-\tau_1)^\frac 79\sum^i_{k=1} \frac 1{m_k}\frac 1{2^{(k+1)\frac 79}}\leq  \ln \frac c{(\tau-\tau_1)^\frac 73},$$
and
$$A_3=\ln \Big(1+\frac {\tau-\tau_1}{\sqrt{\gamma-\gamma_1}}+\Big(\frac 1{\gamma_1\tau_1^3}\Big)^\frac1{10}M(2R)\Big)  \sum^i_{k=1}\frac{1}{m_k}\leq 
$$
$$\leq \ln \Big(1+\frac {\tau-\tau_1}{\sqrt{\gamma-\gamma_1}}+\Big(\frac 1{\gamma_1\tau_1^3}\Big)^\frac1{10}M(2R)\Big)^3.$$
So, 
$$\xi_i\leq\frac c{(\tau-\tau_1)^\frac {16}3}
\Big(1+\frac {\tau-\tau_1}{\sqrt{\gamma-\gamma_1}}+\Big(\frac 1{\gamma_1\tau_1^3}\Big)^\frac1{10}M(2R)\Big)^3.$$
Passing to the limit as $i\to\infty$ in (\ref{iter3}),  we complete the proof the Proposition.
\end{proof}
\begin{remark}
	\label{rem1} 
	If we additionally assume that $\pi(\cdot,-\theta R^2) \geq k$ in $B$ for some $0<\theta\leq 1$, then we do not need 
to use a cut-off in $t$. So, for $0<\lambda<1$, we have 
$$\sup\limits_{Q^{\lambda R,\theta}(R)}\sigma\leq c'_1(\lambda,M(2R))\Big(\frac 1{|Q^{1,\theta}(R)|}\int\limits_{Q^{1,\theta}(R)}\sigma^\frac {10}3dz\Big)^\frac 3{10}, $$ where
$$c'_1(\lambda,M(2R))=\frac{c}{(1-\lambda)^\frac {16}3}\Big(1+\Big(\frac 1{\theta\lambda^3}\Big)^\frac 1{10}M(2R)\Big)^3.$$
\end{remark}
\begin{cor}
	\label{cor1} Let  a non-negative function $\pi\in\mathcal V_0$ have the property $(\mathcal B_R)$ in $Q(2R)$ and let $0<\lambda_1<\lambda<2$ and $0<\theta\leq 1$. Suppose that 
	\begin{equation}
		\label{cor11}
		|\{\pi<k\}\cap Q^{\lambda, \theta}((0,t_0),R)|<\mu |Q^{\lambda,\theta}(R)|
	\end{equation}
	for some $t_0>-4R^2$, for some $0<k\leq k_R$, and for some $$0<\mu\leq\mu_*=\Big(\frac 1{2c_1(\lambda_1,\lambda,\theta/2,\theta,M(2R))}\Big)^\frac {10}3.$$
	Then
	$\pi\geq \frac k2$
	in $Q^{\lambda_1,\theta/2}((0,t_0),R)$.
	
	If, in addition, $\pi(\cdot,t_0-\theta R^2)>k$ in $\mathcal C(\lambda R)$, then
$\pi\geq \frac k2$
	in $Q^{\lambda_1,\theta}((0,t_0),R)$.
	\end{cor}
\begin{proof}
The first statement can be proved ad absurdum with the help of inequality 	(\ref{Mozer1}) and a suitable choice of the number $\mu_*$. The second statement is proved in the same way but with the help of the inequality of Remark \ref{rem1}. Number $\mu_*$ is defined by the constant $c_1'$ instead of $c_1$.
\end{proof}

The two lemmas below are obvious modifications of the corresponding statements in the paper \cite{NazUr}. 
\begin{lemma}
\label{theta0}	
Let $0\leq \pi\in \mathcal V_0$  
 have the property $(\mathcal B_R)$ in $Q(2R)$. Given $\delta_0\in ]0,1]$, there exists a positive number $\theta_0(\delta_0, f(2R))\leq 1$ such that if, for $0<\theta\leq \theta_0$, $0<k_0\leq k_R$, there holds 
 $$|\{\pi(\cdot,t_0-\theta R^2)\geq k_0\}\cap \mathcal C(R)|>\delta_0|\mathcal C(R)|,
 $$ then 
 $$|\{\pi(\cdot,t)\geq\frac {\delta_0}3k_0\}\cap \mathcal C(R)|>\frac {\delta_0}3|\mathcal C(R)|
 $$ for all $t\in [t_0-\theta R^2,t_0]$.
\end{lemma}
\begin{remark}
\label{rem2} There is a formula for $\theta_0$:
$$\theta_0=\Big(\frac{c\delta_0^6} {1+\delta_0^2f(2R)}\Big)^\frac 43. 
$$
\end{remark}

\begin{lemma}
	\label{iterations}Let $0\leq \pi\in \mathcal V_0$  
 have the property $(\mathcal B_R)$ in $Q(2R)$. Let, for any $t\in [t_0-\theta_1R^2,t_0]$,
 $$|\{\pi(\cdot,t)\geq k_1\}\cap \mathcal C(R)|\geq \delta_1|\mathcal C(R)| $$
 for some $0<k_1\leq k_R$ and for some $0<\delta_1\leq 1$ and $0<\theta_1\leq 1$. 
 
 Then, for any $\mu_1\in ]0,1[$, 
  the following inequality is valid:
 $$|\{\pi<2^{-s}k_1\}\cap Q^{1,\theta_1}((0,t_0),R)|\leq \mu_1|Q^{1,\theta_1}(R)|$$
 with  the integer number $s$ defined as
 $$s={\rm entier}   \Big(\frac c{\delta_1^2\mu_1^2\theta_1}(1+f(2R))\Big)+1.$$. 
 \end{lemma}

Given $\theta \in ]0,1]$, we can find an number $0<R_{*1}(c_*,\alpha,\theta))\leq 1$ so that $\Big(\frac 1{cg(2r)}\Big)^\frac 43\leq \theta$ for all $0<r\leq R_{*1}$.
\begin{cor}
\label{cor2} Let $0\leq \pi\in \mathcal V_0$  
 have the property $(\mathcal B_R)$ in $Q(2R)$. If $\pi(\cdot,\overline t)\geq k_2$ in $\mathcal C(R)$, then, for any $\sigma\in]0,1[$, 
the inequality  	$\pi\geq \beta_2k_2$ holds in $Q^{\sigma,\theta_0}((0,t_0),R)$, where
 	$$\beta_2=\frac 16 2^{-c(1-\sigma)^{-40}\sigma^{-6}g^{25}(2R)}
 	$$
 	 provided $R\leq R_{*1}$. 
 	 \end{cor}
\begin{proof} We apply Lemma \ref{theta0} with $\delta_0=\delta_2=1$ and $k_0=k_2$. Then, for $\sigma=4/(27c)$, we calculate 
$$
\theta_0=\Big(\frac {\frac 4{27}\sigma}{\frac 1\sigma+f(2R)}\Big)^\frac 43\geq \Big(\frac c{g(2R)}\Big)^\frac 43$$
and state that the following inequality holds:
$$|\{\pi(\cdot,t)>\frac {k_0}3\}\cap \mathcal C(R)|\geq \frac 13|\mathcal C(R)|$$ for any $t\in [t_0-\theta_0R^2,t_0]$, where $t_0=\overline t+\theta_0R^2$. In what follows, we are going to use the quantity $(c/(g(2R)))^\frac 43$ as a new number $\theta_0$ instead of $\theta_0(1,f(2R))$. 

Now, we are going to apply Lemma \ref{iterations} with another set of parameters $k_1=\frac 13k_2$, $\theta_1=\theta_0$, $\delta_1=\frac 13$, and $$ \mu_1=\mu_*=\Big(\frac 1{2c'_1}\Big)^\frac {10}3, \qquad c'_1=\frac c{(1-\sigma)^\frac {16}3}\Big(1+\Big(\frac 1{\theta_0\sigma^3}\Big)^\frac 1{10}M(2R)\Big)^3\leq $$$$\leq\frac c{(1-\sigma)^\frac {16}3}\Big(\frac 1{\theta_0\sigma^3}\Big)^\frac 3{10}g^3(2R).$$
   Lemma (\ref{iterations}) gives us: 
$$|\{\pi<2^{-s}k_1\}\cap Q^{1,\theta_1}((0,t_0),R)|<\mu_1| Q^{1,\theta_1}(R)|,$$ 
where 
$$s={\rm entier}   \Big(\frac c{\delta_1^2\mu_1^2\theta_1}(1+f(2R))\Big)+1.$$
  But we know that $$\pi(\cdot,t_0-\theta_0 R^2)\geq k_2>2^{-s}k_1=2^{-s}\frac {k_2}3.$$ Then, from Corollary \ref{cor1}, it follows that $\pi>\frac 122^{-s}k_1=\beta_2k_2$ with
$\beta_2=\frac 122^{-s}\frac 13$ in $Q^{\sigma,\theta_0}((0,t_0),R)$. 

\end{proof}
\begin{lemma}\label{mainiter}
Let $0\leq \pi\in \mathcal V_0$  
 have the property $(\mathcal B_R)$ in $Q(2R)$, assuming that $R\leq R_{*1}(c_*,\alpha,\theta)$ for some $0<\theta\leq1$. Suppose further that, for some $0<k\leq k_R$ and for some $-R^2\leq \overline t\leq -\theta R^2$, there holds
 $\pi(\cdot,\overline t)\geq k$ in $\mathcal C(R)$.
 Then $\pi\geq \beta_0k$ in $\widehat Q:=\mathcal C(\frac 23R)\times [\overline t,0]$, where 
 $$\beta_0\geq\ln^{-\frac 12}( 1/R)$$
 for $R\leq R_{*2}(c_*,\alpha,\theta).$
\end{lemma}
\begin{proof} Let 
	$$N={\rm entier}\Big(\frac 98\frac{ |\overline t|}{\tilde \theta_0R^2}\Big)+1,
	$$
where $\tilde \theta_0=(c/g(\frac 232R))^\frac 43
 \leq \theta$. Next, we introduce
$$\hat\theta_0=\frac { |\overline t|}{(\frac {8N}9+\frac 1{2N})R^2}\leq \tilde\theta_0.$$

{\it Step 1}. By Corollary \ref{cor2}, the inequality $\pi \geq 
\beta^{(1)}_2k$ holds at least in $\mathcal C((1-\frac1{3N})R)\times [\overline t_1,\overline t_1+\hat\theta_0R^2]$, where $\overline t_1=\overline t$, $\overline t_2= \overline t_1+\hat\theta_0R^2$, $\sigma =1-1/(3N)\geq 2/3$, $1-\sigma=1/(3N)$, and 
$$\ln \beta^{(1)}_2=-\ln 6-cN^{40}g^{25}(2R)
$$

{\it Step 2}. Here, we are going to use Corollary \ref{cor2} with $R(1-1/(3N))$ instead of $R$ and with $\sigma =(1-2(3N))/(1-1/(3N))$. As a result, we have the estimate $\pi\geq 
\beta^{(2)}_2\beta^{(1)}_2k$ at least in $\mathcal C((1-2/(3N))R)\times [\overline t_2,\overline t_2+\hat\theta_0 (1-1/(3N))^2R^2]$, $\overline t_3=\overline t_2+\hat\theta_0 (1-1/(3N))^2R^2$, and 
$$\ln \beta^{(2)}_2=-\ln 6-cN^{40}g^{25}(2(1-1/(3N))R).
$$ So, $\pi\geq 
\beta^{(2)}_2\beta^{(1)}_2k$ in $\mathcal C((1-2(3N))R)\times [\overline t, \overline t_3]$.

After $N$ steps, we shall have $\overline t_N=0$ and 
$$\pi\geq \beta^{(N)}_2...\beta^{(1)}_2k=\beta_0(R)k$$
in $\mathcal C(\frac 23R)\times [\overline t,0]$, where
$$\ln \beta^{(i+1)}_2=-\ln 6-cN^{40}g^{25}(2(1-i/(3N))R)
$$
for $i=0,1,...,N-1$.

Next, according to assumption (\ref{sca1}), we can have
$$\ln \beta_0\geq -N\ln 6-cN^{40}\sum\limits^{N-1}_{k=1} c_*^{25}\ln^\gamma \ln^\frac12\Big(\frac 1{2(1-i/(3N))R\Big)},
$$
where $25\alpha<1$. Since
$$\ln \frac 1{1-x}\leq 2x
$$
provided $0\leq x\leq 1/2$, we find, assuming that $R\leq 1/6$, the following:
$$\ln^\gamma\ln^\frac12\Big(\frac 1{2(1-i/(3N))R\Big)}\leq \ln^\gamma\Big(\ln \frac{1}{2R} +\frac iN\Big)^\frac 12  \leq    
$$
$$\leq\ln^\gamma\Big(\ln^\frac 12 \frac{1}{2R} +\Big(\frac iN\Big)^\frac 12\Big)= \ln^\gamma\Big(\ln^\frac 12 \frac{1}{2R}\Big(1 +\Big(\frac i{N\ln \frac 1{2R}}\Big)^\frac 12\Big)\leq$$
$$\leq\ln^\gamma\Big(\ln^\frac 12 \frac{1}{2R}\Big(1 +\Big(\frac i{N}\Big)^\frac 12\Big)=\Big(\ln\Big(\ln^\frac 12 \frac{1}{2R}\Big)+\ln\Big(1 +\Big(\frac i{N}\Big)^\frac 12\Big)\Big)^\gamma\leq$$
$$\leq \Big(\ln\Big(\ln^\frac 12 \frac{1}{2R}\Big)+\Big(\frac i{N}\Big)^\frac 12\Big)^\gamma\leq\ln^\gamma\Big(\ln^\frac 12 \frac{1}{2R}\Big)+\Big(\frac i{N}\Big)^\frac \gamma 2.$$

From the latter inequality, one can deduce the bound
$$\ln \beta_0\geq -N\ln6-cc_*^{25}N^{40}\Big(N\ln^\gamma\ln^\frac 12 \frac{1}{2R}+\sum\limits^{N-1}_{i=0}\Big(\frac iN\Big)^\frac \gamma2\Big)\geq$$$$\geq-N\ln6-cc_*^{25}N^{41}\ln^\gamma\ln^\frac 12 \frac{1}{2R},$$
which is valid for $0<R\leq R_{*3}(\alpha)\leq 1/6$.
Taking into account that $N\leq c(g(2R))^\frac 43$, we conclude
$$\ln \beta_0\geq -c_1(c_*)\ln^\frac {239\alpha}3\sqrt{\ln \frac 1R}.$$
It remains to find $R_{*4}(c_*,\alpha)\leq 1$ such that
$$c_1(c_*)\ln^{\frac {239\alpha}3-1}\sqrt{\ln \frac 1R}\leq 1$$
for all $0<R\leq R_{*4}$. So, we have the required inequality provided $0<R\leq R_{*2}=\min \{R_{*1},R_{*3},R_{*4}\}$.\end{proof}

\setcounter{equation}{0}
\section{Proof of Proposition \ref{modcon}}
Now, we can state an analog  of Lemma 4.2 of \cite{NazUr} for the class $\mathcal V$.
\begin{lemma}
	\label{analog}
	Let $0\leq \pi\in\mathcal V_0$ possess the property $(\mathcal B_R)$ in $Q(2R)$. 	
	
	Suppose further that \begin{equation}
		\label{2nd}\pi\leq M_0k_R\end{equation}
in $Q(2R)$ for some $M_0\geq 1$.		
Then, there exists $\overline t\in [-R^2,-\frac 34R^2]$ such that
		\begin{equation}
			\label{consequence1
			}
			|e_{\kappa_0}(\overline t)|\geq \delta_0|B(R)|
		\end{equation}
		Here,  $\kappa_0=\kappa_0(f(2R))= c/(1+f(2R))$,
		$e_\kappa(t):=\{x\in \mathcal C(R): \pi(x,t)\geq\kappa k_R\}$,
		and $$ \delta_0(M_0,f(2R))=\Big(\frac c{M_0(1+f(2R))}\Big)^\frac94.$$
\end{lemma}

\begin{proof}
Here, we 
follow arguments of 
the paper \cite{NazUr}. They are based on the identity:
$$\int\limits_{Q}(-\pi\partial_t\eta- \pi\Delta \eta-(v+2x'/|x'|^2)\cdot\nabla \eta\pi)dxdt= $$
\begin{equation}\label{crucialidentity}
	=4\pi_0
\int\limits_{-1}^0\int\limits_{-1}^1
\pi(0,x_3,t)\eta(0,x_3,t) dx_3dt,
\end{equation}
which is valid for any  non-negative test function $\eta$  supported in $Q$. Here, $\pi_0=3.14...$.  Although 
a similar statement 
has been proven in \cite{NazUr} under the assumption  that $\pi$ is Lipschitz, it remains to be 
true for functions $\pi$ from the class $\mathcal V_0$ as well. Indeed, take a smooth cut-off function $\psi=\psi(x')$ so that $\psi(x')=\Psi(|x'|)$, $0\leq \psi\leq 1$, $\psi(x')=0$ if $|x'|\leq \varepsilon/2$, $\psi(x')=1$ if $|x'|\geq \varepsilon$, $\Psi'(\varrho)\leq c/\varrho$ and $\Psi''(\varrho)\leq c/\varrho^2$ for some positive constant $c$. Then, it follows from (\ref{GammaEq2}) that: 
$$\int\limits_{Q}
\Big(\pi\partial_t(\eta\psi)+\pi(u+b)\cdot\nabla(\eta\psi)+\pi\Delta(\eta\psi)\Big)dz=0.
$$
There are two difficult terms for passing to the limit as $\varepsilon\to0$. The first one is as follows:
$$I_1:=\int\limits_{Q}\pi\eta\Delta \psi dx dt=J_1+J_2,$$
where
$$J_1:=\int\limits_{Q}(\pi\eta-(\pi\eta)|_{x'=0})\Delta\psi dx dt,
$$
For $J_2$, we find
$$J_2:=\int\limits_{Q}(\pi\eta)|_{x'=0}\Delta\psi dx dt=\int\limits^0_{-1}\int\limits^1_{-1}(\pi\eta)|_{x'=0}dx_3dt\int\limits_{|x'|<1}\Delta \psi(x')dx'$$
and 
$$\int\limits_{|x'|<1}\Delta \psi(x')dx'=2\pi_0\int\limits^\varepsilon_\frac \varepsilon 2\frac 1\varrho\frac \partial{\partial \varrho}\Big (\varrho\Psi'(\varrho)
\Big )\varrho d\varrho=2\pi_0\varrho\Psi'(\varrho)\Big|^\varepsilon_\frac \varepsilon2=0.$$
Now, we wish to show that 
$$J_1:=\int\limits_{Q}\xi\Delta\psi dx dt \to0$$
as $\varepsilon\to0$, where, 
$\xi:=\pi\eta-(\pi\eta)|_{x'=0}.$
To this end, let us introduce the function
$$H_\varepsilon(x_3,t):=\int\limits_{\frac \varepsilon 2 <\varrho<\varepsilon}\xi\Delta \psi dx'.$$ 
It can be bounded from above and from below
$$|H_\varepsilon(x_3,t)|\leq c \sup\limits_{{\rm spt} \eta}\pi\sup\limits_{|x'|<1}\eta(x',x_3,t)\frac 1{\varepsilon^2}\int\limits_{\frac \varepsilon2}^\varepsilon\varrho d\varrho=:h(x_3,t)$$
provided $\varepsilon<1$.
The function $h$ is supported in $]-1,1[ \times ]-1,0[$ and thus 
$$\int\limits_{-1}^1\int\limits^0_{-1}h(x_3,t)dx_3dt<\infty.
$$

Now, let $(0,x_3,t)$ be a regular point of $\pi$, i.e., $(0,x_3,t)\notin S^\pi$. Then, 
$\xi(x',x_3,t)\to 0$ as $|x'|\to 0$ and thus for any $\delta>0$ there exists a number $\tau(x_3,t)>0$ such that $|\xi(x',x_3,t)|<\delta$ provided $|x'|<\tau$. So, 
$$|H_\varepsilon(x_3,t)|<c\frac \delta{\varepsilon^2}\int\limits^\varepsilon_\frac \varepsilon2\varrho d\varrho=c\frac \delta  2$$
provided $\varepsilon<\tau$. Therefore, $H_\varepsilon(x_3,t)\to 0$ as $\varepsilon\to0$ and by the Lebesgue theorem on dominated convergence, we find that $$J_1=\int\limits_{-1}^1\int\limits_{-1}^0H_\varepsilon(x_3,t)dx_3dt\to0$$
as $\varepsilon\to0$.

Similar arguments work for the second difficult term:
$$I:=\int\limits_{Q}\pi\eta b\cdot \nabla \psi dz=J_1+J_2,$$
where 
$$J_1=\int\limits_{Q}\xi b\cdot \nabla \psi dz $$
and 
$$J_2:=\int\limits_{Q}(\pi\eta)|_{x'=0}b\cdot \nabla \psi dx dt=\int\limits^0_{-1}\int\limits^1_{-1}(\pi\eta)|_{x'=0}dx_3dt2\pi_0\int\limits_\frac \varepsilon2^\varepsilon  \frac 2\varrho\Psi'(\varrho)\varrho d\varrho=$$
$$=4\pi_0\int\limits^0_{-1}\int\limits^1_{-1}(\pi\eta)|_{x'=0}dx_3dt.$$

The fact that $J_1\to0$ as $\varepsilon\to0$ can be justified in the same way as above, replacing $H_\varepsilon$ with the function
$$G_\varepsilon(x_3,t):=\int\limits_{\frac \varepsilon2<|\
x'|
<\varepsilon}\xi b\cdot\nabla \psi dx'.
$$

Other terms can be treated in a similar way and even easier. So,  the required identity (\ref{crucialidentity}) has been proven.

Now, let us select the test function $\eta$ in (\ref{crucialidentity}), using the following notation
$$Q^{\lambda,\theta}(z_0,R):=\mathcal C(x_0,\lambda R)\times]t_0-\theta R^2,t_0[,
$$
so that $\eta=1$ in $Q^{\frac 12,\frac 18}((0,-\frac {13}{16}R^2),R)$, $\eta=0$ out of $Q^{1,\frac 14}((0,-\frac 34R^2),R)$ and 
$|\partial_t\eta|+|\nabla \eta|^2+|\nabla^2\eta|\leq c/R^2$. Taking into account that $\pi$ has the property $(\mathcal B_R)$, we find
$$\frac {\pi_0}2k_RR^2\leq \frac c{R^2}\int\limits_{Q^{1,\frac 14}(z_R,R)}\pi dz+\frac cR\int\limits_{Q^{1,\frac 14}(z_R,R)}\pi |v|dz+\frac cR\int\limits_{Q^{1,\frac 14}(z_R,R)}\frac \pi {|x'|}dz,$$
where $z_R=(0,-\frac 34R^2)$.

Setting
$ E_\kappa=\{(x,t): t\in ]-R^2,-\frac 34R^2[, x\in e_\kappa(t)\},$ we can deduce from the latter inequality 
$$\frac {\pi_0}2k_RR^3\leq$$
$$\leq \frac c{R^2}\int\limits_{Q^{1,\frac 14}(z_R,R)\setminus E_\kappa}\pi dz+\frac cR\int\limits_{Q^{1,\frac 14}(z_R,R)\setminus E_\kappa}\pi |v|dz+\frac cR\int\limits_{Q^{1,\frac 14}(z_R,R)\setminus E_\kappa}\frac \pi {|x'|}dz+$$
$$+ \frac c{R^2}\int\limits_{Q^{1,\frac 14}(z_R,R)\cap E_\kappa}\pi dz+\frac cR\int\limits_{Q^{1,\frac 14}(z_R,R)\cap E_\kappa}\pi |v|dz+\frac cR\int\limits_{Q^{1,\frac 14}(z_R,R)\cap E_\kappa}\frac \pi {|x'|}dz.$$
Applying (\ref{2nd}) and recalling definitions of the sets $e_\kappa(t)$ and $E_\kappa$, we can get 
$$\frac {\pi_0}2k_RR^3\leq$$
$$\leq \frac {c\kappa k_R}{R^2}\Big\{|Q^{1,\frac 14}(R)|
+R\int\limits_{Q^{1,\frac 14}(z_R,R)\setminus E_\kappa} |v|dz+R\int\limits_{Q^{1,\frac 14}(z_R,R)\setminus E_\kappa}\frac 1 {|x'|}dz\Big\}+$$
$$+\frac {cM_0k_R}{R^2}\Big\{|E_\kappa|+R\int\limits_{Q^{1,\frac 14}(z_R,R)\cap E_\kappa} |v|dz+R\int\limits_{Q^{1,\frac 14}(z_R,R)\cap E_\kappa}\frac 1 {|x'|}dz\Big\}.$$

We need to estimate integrals in the above inequality. First, for integrals, containing $v$, Holder inequality gives
$$\int\limits_{Q^{1,\frac 14}(z_R,R)\setminus E_\kappa}|v|dx\leq
\|\mathbb I\|_{\frac 32,\frac 43,Q^{1,\frac 14}(R)}\Big(\int\limits^{-\frac 34R^2}_{-R^2}\Big(\int\limits_{\mathcal C(R)}|v|^3dx\Big)^\frac 43dt\Big)^\frac 14\leq $$
$$\leq f(2R)R^\frac 12\|\mathbb I\|_{\frac 32,\frac 43,Q^{1,\frac 14}(R)}\leq f(2R)R^4$$
and similarly 
$$\int\limits_{Q^{1,\frac 14}(z_R,R)\cap E_\kappa} |v|dz\leq f(2R)R^\frac 12\|\mathbb I\|_{\frac 32,\frac 43,E_\kappa}.$$

To evaluate the last two integrals, let us take into account the fact:
$$\frac 1{|x'|}\in L_{\frac 95,\infty}(Q^{1,\frac 14}(z_R,R)).$$
Then,
$$ \int\limits_{Q^{1,\frac 14}(z_R,R)\setminus E_\kappa}\frac 1 {|x'|}dz\leq \|\frac 1{|x'|}\|_{\frac 95,\infty,Q^{1,\frac 14}(z_R,R)}\|\mathbb I\|_{\frac 94,1,Q^{1,\frac 14}(R)}\leq 
$$
$$\leq cR^\frac 23R^\frac {10}3=cR^4
$$
and 
$$\int\limits_{Q^{1,\frac 14}(z_R,R)\cap E_\kappa}\frac 1 {|x|}dz\leq c R^\frac 23\|\mathbb I\|_{\frac 94,1,E_\kappa}.$$

Hence, we have 
$$\frac {\pi_0}2k_RR^3\leq c\kappa k_RR^3(1+f(2R))+$$
$$+\frac {cM_0k_R}{R^2}\Big[|E_\kappa|+f(2R)R^\frac 32\|\mathbb I\|_{\frac 32,\frac 43,E_\kappa}+R^\frac 53\|\mathbb I\|_{\frac 94,1,E_\kappa}\Big].$$
So,
$$\frac {\pi_0}2\leq c\kappa (1+f(2R))
+\frac {cM_0}{R^5}\Big[|E_\kappa|+f(2R)R^\frac 32\|\mathbb I\|_{\frac 32,\frac 43,E_\kappa}+R^\frac 53\|\mathbb I\|_{\frac 94,1,E_\kappa}\Big].$$
Now, one can find $\kappa=\kappa_0(f(2R))=c/(1+f(2R))$ such that
$$\frac {cM_0}{R^5}\Big[|E_{\kappa_0}|+f(2R)R^\frac 32\|\mathbb I\|_{\frac 32,\frac 43,E_{\kappa_0}}+R^\frac 53\|\mathbb I\|_{\frac 94,1,E_{\kappa_0}}\Big]\geq 1.$$
It remains to estimate two integrals on the left hand side of the latter inequality:
$$\|\mathbb I\|_{\frac 32,\frac 43,E_{\kappa_0}}=\Big(\int\limits^{-\frac 34R^2}_{-R^2}|e_\kappa(t)|^\frac 89dt\Big)^\frac 34\leq c|E_{\kappa_0}|^\frac 23R^\frac 16
$$ and
$$\|\mathbb I\|_{\frac 94,1,E_\kappa}\leq c|E_{\kappa_0}|^\frac 49R^\frac {10}9.
$$
Letting $A=|E_{\kappa_0}|/R^5$, we arrive at the following inequality
$$f(A):=A+A^\frac 49+f(2R)A^\frac 23\geq \frac 1{cM_0}.$$
Since $f'(A)>0$ for $A>0$, we can state that the last inequality implies
$$\frac {|E_{\kappa_0}|}{|\mathcal C(R)|\frac 14R^2}\geq\delta_0=\Big(\frac c{M_0(1+f(2R))}\Big)^\frac 94.$$

 It is not so difficult to show the exisence 
of $\overline t\in [-R^2,-\frac 34R^2]$  with the property:
$$|e_{\kappa_0}(\overline t)|\frac 14R^2\geq |E_{\kappa_0}|.$$
So, it is proven that
there exists $\bar t\in [-R^2,-3R^2/4]$ such that
\begin{equation}
	\label{basic ineq}|\{x\in \mathcal C(R): \pi(x,\bar t)>\kappa_0 k_R\}|\geq \delta_0 |\mathcal C(R)|,
\end{equation}
which completes the proof of the lemma.
\end{proof}

Now, we are able to  prove Proposition \ref{modcon}.

 Assume that the function $\pi$ meets all the conditions of Lemma \ref{analog} and  according to it, we can claim that: 
$$|e_{\kappa_0}(\overline t)|=|\{x\in \mathcal C(R): \pi(x,\overline t)\geq\kappa_0 k_R\}|\geq \delta_0|\mathcal C(R)|$$	
for some $\overline t\in [-R^2,-\frac 34R^2]$, $\kappa_0=c/g(2R)$, and $\delta_0=c(M_0)/g^\frac 94(2R)$. Now, we can  calculate
$$\theta(\delta_0(M_0,f(2R)),f(2R))\geq c\Big(\frac {\delta_0^6}{1+\delta_0^2f(2R)}\Big)^\frac 43\geq $$
$$\geq c(M_0)\Big(\frac 1{g(2R)}\Big)^{18},$$ 
apply Lemma \ref{theta0}, and find
$$|\{\pi(\cdot,t)\geq \delta_0\kappa_0k_R/3\}\cap\mathcal C(R)|>\delta_0/3|\mathcal C(R)|$$
for all $t\in [\overline t,t_0]$ with $t_0=\overline t+\theta_0R^2$ and $\theta_0=c(M_0)(g(2R))^{-18}$.

Next, it follows from Lemma \ref{iterations} that: 
$$|\{\pi<2^{-s}\delta_0\kappa_0k_R/3\}\cap Q^{1,\theta_0}((0,t_0),R)|\leq \mu_*
|Q^{1,\theta_0}(R)|,$$ where
$$s={\rm entier}\Big(\frac {c}{\delta_0^2\mu_*^2\theta_0}(1+f(2R))\Big)+1$$
and $\mu_*$ is the number that appears in Corollary \ref{cor1}, see also Proposition \ref{Mozer}.  In our case,
$$\mu_*=\Big(\frac 1{2c_1(3/4,1,\theta_0/2,\theta_0,M(2R))}\Big)^\frac {10}3$$
and, moreover 
$$c_1(3/4,1,\theta_0/2,\theta_0,M(2R))\leq c\theta_0^{-\frac32}g^3(2R)\leq c(M_0)(g(2R))^{30}.
$$
Then, Corollary \ref{cor1} implies the bound
$$\pi\geq2^{-s}\delta_0\kappa_0k_R/6=\hat\beta_2\kappa_0k_R$$
in $Q^{\frac 34,\frac 12\theta_0}((0,t_0),R)$.
So, combining previous estimates, we find the following: 
$$\hat\beta_2=\frac 162^{-s}\delta_0\geq e^{-sln2-\ln6}\delta_0\geq e^{-cs}\delta_0,
$$
where
$$s\leq \frac {2g(2R)}{\delta_0^2\mu_*^2\theta_0}\leq c(M_0)g(2R)(g(2R))^\frac 92(g(2R))^{18}{c_1}^\frac {20}3\leq$$
$$\leq c(M_0)(g(2R))^\frac {47}2(g(2R))^{30})^\frac {20}3\leq c(M_0) (g(2R))^{224}.
$$
So, $$\hat\beta_2\geq e^{-c(M_0)(g(2R))^{224}}c(M_0)(g(2R))^{-\frac 94}\geq e^{-2c(M_0)(g(2R))^{224}}\geq$$$$
\geq e^{-c(M_0,c_*)\ln^{224\alpha}\sqrt{\ln \frac1{R}}}.$$
Obviously, there exists a number $0<R_{*5}(M_0,c_*,\alpha)\leq \min\{1/6,R_{*2}\}$ such that
$$2c(M_0,c_*)\ln^{224\alpha-1}\sqrt{\ln \frac1{R}}\leq 1$$
and 
$$c(M_0,c_*)\ln^{224\alpha}\sqrt{\ln \frac1{R}}\geq \ln\ln^\alpha\sqrt{\ln \frac 1R}$$
for $0<R\leq R_{*5}(M_0,c_*,\alpha)$ and thus
$$-c(M_0,c_*)\ln^{224\alpha}\sqrt{\ln \frac1R}=-2c(M_0,c_*)\ln^{224\alpha}\sqrt{\ln \frac1R}+$$$$+c(M_0,c_*)\ln^{224\alpha}\sqrt{\ln \frac1R}\geq -\ln\sqrt{\ln \frac1R}+\ln\ln^\alpha\sqrt{\ln \frac 1R}.$$
Now, the number $\hat\beta_2$ is estimated as follows:
\begin{equation}
	\label{hatbeta2}
\hat\beta_2\geq \Big(\ln\frac 1R
\Big)^{-\frac 12}\ln\ln^\alpha\sqrt{\ln \frac 1R}
\end{equation} for $0<R\leq R_{*5}(M_0,c_*,\alpha)$. 

Since $$-R^2\leq \overline t+\theta_0/2R^2=t_0-\theta_0/2R^2<t_0=\overline t+\theta_0R^2\leq -\frac 34R^2+\frac 14R^2=-\frac 12R^2,$$
 there is  $\overline t_1\in [-R^2,-\frac 12R^2]$ such that
$$
\pi(\cdot,\overline t_1)>\hat\beta_2\kappa_0k_R$$
in $\mathcal C(\frac 34R)$. 
It allows us to apply Lemma \ref{mainiter} with $\theta=1/2$, with $\frac 34R$ instead of $R$, with $\overline t_1$ instead of $\overline t$, and with $\hat\beta_2\kappa_0k_R$ instead of $k$.  According to Lemma \ref{mainiter}, the inequality 
$$\pi \geq \beta_0\hat\beta_2\kappa_0k_R$$
holds in $Q(R/2)$. It follows from Lemma \ref{mainiter} and from (\ref{hatbeta2}) that
$$\pi \geq \frac {c(c_*)k_R}{\ln (\frac 1R)}
=\beta(2R)k_R$$
in $Q(R/2)$.

By our assumption imposed on function $\sigma$, we can put $k_R=\frac 12{\rm osc}_{z\in Q(2R)}\sigma(z)$. 
Then, either $\pi=\sigma-m_{2R}$ or $\pi=M_{2R}-\sigma(z)$ satisfies all the conditions of the proposition with $M_0=2$. Simple arguments show that
$${\rm osc}_{z\in Q(R/2)}\sigma(z)\leq \Big(1-\frac 12\beta(2R)\Big){\rm osc}_{z\in Q(2R)}\sigma(z).
$$
Now, after iterations of the latter inequality, we arrive at the following bound
$${\rm osc}_{z\in Q(R/2^{2k+1})}\leq \prod\limits^{k}_{i=0}\Big((1-\frac 12\beta(R/2^{2k+1})\Big){\rm osc}_{z\in Q(2R)}\sigma(z)=$$$$=\eta_k{\rm osc}_{z\in Q(2R)}\sigma(z)$$
being valid for any natural number $k$.

In order to evaluate $\eta_k$, take $\ln$  of it. As a result, 
$$\ln \eta_k=\sum\limits^k_{i=0}\ln \Big((1-\frac 12\beta(R/2^{2k+1})\Big)\leq -\sum\limits^k_{i=0}\frac 12\beta(R/2^{2k+1})=$$
$$=-c(c_*)\sum\limits^k_{i=0}(\ln ( 2^k/R))^{-1}
=-c(c_*)\sum\limits^k_{i=0}\frac 1{k\ln 2+\ln 1/R
}\leq$$
$$\leq-c(c_*)\int\limits^{k+1}_0\frac {dx}{x\ln 2+\ln 1/R
}=$$$$=-c(c_*)\Big(\ln(2^{k+1}/R)
-\ln(1/R)\Big)=-c(c_*)\Big(\ln(2^{k+1})
\Big).
$$
So, (\ref{osc}) follows. The proof of  Proposition \ref{modcon} is complete.

\setcounter{equation}{0}
\section{Proof of Theorem \ref{mainresultglobal}}

By the maximimum principle, we have 
$|\sigma| =|\varrho v_\varphi|\leq \Sigma_0$
in $\mathbb R^3\times ]0,T[$. From Proposition \ref{modcon}, it follows that
$$|\sigma(\varrho,x_3,t)|\leq C_1(c_*)\Big(\frac \varrho{2R_*}\Big)^{C_2(c_*)} 2\Sigma_0
$$ fo all $0<\varrho\leq R_*(c_*,\alpha)$, for all $x_3\in\mathbb R$, and for $t\in ]T-R_*^2,T[$.
For $\varrho>R_*$, we simply have 
$$|\sigma(\varrho,x_3,t)|\leq \Sigma_0\Big(\frac \varrho{R_*}\Big)^{C_2(c_*)}.$$
It remains to notice that $v(\cdot, T-R_*^2)\in H^2$. Therefore, one can use the main result
of the paper \cite{ChenFangZhang2017}, see also \cite{MiaoZheng2013} and \cite{LeiZhang2017}, for the Cauchy problem for the Navier-Stokes system (\ref{NSE}) in $\mathbb R^3\times  ]T-R_*^2,T[$ and conclude that $v$ is a strong solution in the interval $]0,T[$.

\end{document}